\documentclass[reqno]{amsart}
\usepackage{amssymb}
\usepackage{ifpdf}
\ifpdf
  \usepackage[hyperindex,pagebackref]{hyperref}
\else
  \expandafter\ifx\csname dvipdfm\endcsname\relax
    \usepackage[hypertex,hyperindex,pagebackref]{hyperref}
  \else
    \usepackage[dvipdfm,hyperindex,pagebackref]{hyperref}
  \fi
\fi
\theoremstyle{plain}
\newtheorem{thm}{Theorem}[section]
\newtheorem{lem}{Lemma}[section]
\newtheorem{prop}{Proposition}[section]
\theoremstyle{remark}
\newtheorem{rem}{Remark}[section]
\numberwithin{equation}{section}
\allowdisplaybreaks[4]
\DeclareMathOperator{\td}{d\mspace{-2mu}}

\begin{document}

\title[Some inequalities involving Seiffert and other means]
{Some sharp inequalities involving Seiffert and other means and their concise proofs}

\author[W.-D. Jiang]{Wei-Dong Jiang}
\address[Jiang]{Department of Information Engineering, Weihai Vocational University, Weihai City, Shandong Province, 264210, China}
\email{\href{mailto: W.-D. Jiang <jackjwd@163.com>}{jackjwd@163.com}}

\author[F. Qi]{Feng Qi}
\address[Qi]{School of Mathematics and Informatics\\ Henan Polytechnic University\\ Jiaozuo City, Henan Province, 454010\\ China; Department of Mathematics\\ School of Science\\ Tianjin Polytechnic University\\ Tianjin City, 300387\\ China}
\email{\href{mailto: F. Qi <qifeng618@gmail.com>}{qifeng618@gmail.com}, \href{mailto: F. Qi <qifeng618@hotmail.com>}{qifeng618@hotmail.com}, \href{mailto: F. Qi <qifeng618@qq.com>}{qifeng618@qq.com}}
\urladdr{\url{http://qifeng618.wordpress.com}}



\thanks{The first author was partially supported by the Project of Shandong Province Higher Educational Science and Technology Program under grant No. J11LA57}

\begin{abstract}
In the paper, by establishing the monotonicity of some functions involving the sine and cosine functions, we provide concise proofs of some known inequalities and find some new sharp inequalities involving the Seiffert, contra-harmonic, centroidal, arithmetic, geometric, harmonic, and root-square means of two positive real numbers $a$ and $b$ with $a\ne b$.
\end{abstract}

\maketitle

\section{Introduction}

It is well known that the quantities
\begin{align}
C(a,b)&=\frac{a^2+b^2}{a+b}, & \overline{C}(a,b)&=\frac{2(a^2+ab+b^2)}{3(a+b)}, &A(a,b)&=\frac{a+b}{2}, \\ G(a,b)&=\sqrt{ab}\,, & H(a,b)&=\frac{2ab}{a+b}, & S(a,b)&=\sqrt{\frac{a^2+b^2}{2}}\,
\end{align}
are called the contra-harmonic, centroidal, arithmetic, geometric, harmonic, and root-square means of two positive real numbers $a$ and $b$ with $a\ne b$.
\par
For $a,b>0$ with $a\ne b$, the first Seiffert mean $P(a,b)$ was defined in~\cite{back8} by
\begin{equation}\label{eq1.1}
P(a,b)=\frac{a-b}{4\arctan\sqrt{\frac{a}{b}}-\pi}.
\end{equation}
Its equivalent form
\begin{equation}\label{eq1.2}
    P(a,b)=\frac{a-b}{2\arcsin \bigl(\frac{a-b}{a+b}\bigr)}
\end{equation}
was given by~\cite[Eq.~(2.4)]{back9}.
For $a,b>0$ with $a\ne b$, the second Seiffert mean $T(a,b)$ was introduced in~\cite{back10} by
\begin{equation}\label{eq1.3}
    T(a,b)=\frac{a-b}{2\arctan\bigl(\frac{a-b}{a+b}\bigr)}.
\end{equation}

Recently, the following double inequalities involving the Seiffert, contra-harmonic, centroidal, arithmetic, geometric, harmonic, and root-square means of two positive real numbers $a$ and $b$ with $a\ne b$ were obtained.

\begin{prop}[{\cite[Theorem~2.1]{back1}}]\label{th1.1}
The double inequality
\begin{equation}\label{th1.1-ineq}
\alpha A(a,b)+(1-\alpha)H(a,b)<P(a,b)<\beta A(a,b)+(1-\beta)H(a,b)
\end{equation}
holds for all $a,b>0$ with $a\ne b$ if and only if $\alpha\le \frac{2}{\pi}$ and $\beta\ge \frac{5}{6}$.
\end{prop}

\begin{prop}[{\cite[Theorem~2.2]{back2}}]\label{th1.2}
The double inequality
\begin{equation}\label{th1.2-ineq}
\alpha C(a,b)+(1-\alpha)H(a,b)<P(a,b)<\beta C(a,b)+(1-\beta)H(a,b)
\end{equation}
holds for all $a,b>0$ with $a\ne b$ if and only if $\alpha\le \frac{1}{\pi}$ and $\beta\ge \frac{5}{12}$.
\end{prop}

\begin{prop}[{\cite[Theorem~2.1]{back3}}]\label{th1.3}
The double inequality
\begin{equation}\label{th1.3-ineq}
\alpha S(a,b)+(1-\alpha)A(a,b)<T(a,b)<\beta S(a,b)+(1-\beta)A(a,b)
\end{equation}
holds for all $a,b>0$ with $a\ne b$ if and only if $\alpha\le \frac{4-\pi}{(\sqrt{2}\,-1)\pi}$ and $\beta\ge \frac{2}{3}$.
\end{prop}

\begin{prop}[{\cite[Theorem~2.1]{back5}}]\label{th1.4}
The double inequality
\begin{equation}\label{th1.4-ineq}
\alpha\overline{C}(a,b)+(1-\alpha)H(a,b)<P(a,b)<\beta \overline{C}(a,b)+(1-\beta)H(a,b)
\end{equation}
holds for all $a,b>0$ with $a\ne b$ if and only if $\alpha\le \frac{3}{2\pi}$ and $\beta\ge \frac{5}{8}$.
\end{prop}

For more information on this topic, please refer to~\cite{back17, back15, back16, back4, back7, back6, back18, Seiffert-Jiang-Qi.tex, back12, back14}.
\par
We point out that all the proofs of Propositions~\ref{th1.1} to~\ref{th1.4} are very complicated and tedious.
\par
In the paper, by establishing the monotonicity of some functions involving the sine and cosine functions, we provide concise proofs of Propositions~\ref{th1.1} to~\ref{th1.4} and find some new sharp inequalities involving the Seiffert, contra-harmonic, arithmetic, geometric, harmonic, and root-square means of two positive real numbers $a$ and $b$ with $a\ne b$.

\section{Lemmas}

For establishing the monotonicity of some functions involving the sine and cosine functions, we need some lemmas below.

\begin{lem}\label{B-2q-posit}
The Bernoulli numbers $B_{2n}$ for $n\in\mathbb{N}$ have the property
\begin{equation}\label{|B_{2n}|}
(-1)^{n-1}B_{2n}=|B_{2n}|,
\end{equation}
where the Bernoulli numbers $B_i$ for $i\ge0$ are defined by
\begin{equation}
\frac{x}{e^x-1}=\sum_{i=0}^\infty \frac{B_i}{n!}x^i =1-\frac{x}2+\sum_{i=1}^\infty B_{2i}\frac{x^{2i}}{(2i)!}, \quad \vert x\vert <2\pi.
\end{equation}
\end{lem}

\begin{proof}
In~\cite[p.~16 and p.~56]{aar}, it is listed that for $q\ge1$
\begin{equation}\label{zeta(2q)-B(2q)}
\zeta(2q)=(-1)^{q-1}\frac{(2\pi)^{2q}}{(2q)!}\frac{B_{2q}}2,
\end{equation}
where $\zeta$ is the Riemann zeta function defined by
\begin{equation}
  \zeta(s)=\sum_{n=1}^\infty\frac1{n^s}.
\end{equation}
From~\eqref{zeta(2q)-B(2q)}, the formula~\eqref{|B_{2n}|} follows.
\end{proof}

\begin{lem}\label{lem2.1}
For $0<|x|<\pi$, we have
\begin{equation}\label{eq2.1}
    \frac1{\sin x}=\frac1x+\sum_{n=1}^{\infty} \frac{2\bigl(2^{2n-1}-1\bigr)|B_{2n}|}{(2n)!}x^{2n-1}.
\end{equation}
\end{lem}

\begin{proof}
This is an easy consequence of combining the equality
\begin{equation}\label{4.3.68}
\csc x=\frac1x +\sum_{n=1}^{\infty} \frac{(-1)^{n-1}2\bigl(2^{2n-1}-1\bigr)B_{2n}}{(2n)!}x^{2n-1},
\end{equation}
see~\cite[p.~75, 4.3.68]{abram}, with Lemma~\ref{B-2q-posit}.
\end{proof}

\begin{lem}[{\cite[p.~75, 4.3.70]{abram}}]\label{lem2.0}
For $0<|x|<\pi$,
\begin{equation}\label{eq2.0}
   \cot x=\frac{1}{x}-\sum_{n=1}^{\infty}\frac{2^{2n}|B_{2n}|}{(2n)!}x^{2n-1}.
\end{equation}
\end{lem}

\begin{lem}\label{lem2.2}
For $0<|x|<\pi$,
\begin{equation}\label{eq2.2}
\frac{1}{\sin^2x}=\frac{1}{x^2} +\sum_{n=1}^{\infty}\frac{2^{2n}(2n-1)|B_{2n}|}{(2n)!}x^{2(n-1)}.
\end{equation}
\end{lem}

\begin{proof}
Since
$$
\frac{1}{\sin ^2x}=\csc^2x=-\frac{\td}{\td x}(\cot x),
$$
the formula~\eqref{eq2.2} follows from differentiating~\eqref{eq2.0}.
\end{proof}

\begin{lem}\label{lem2.5}
Let $f$ and $g$ be continuous on $[a,b]$ and differentiable in $(a,b)$ such that $g'(x)\ne0$ in $(a,b)$. If $\frac{f'(x)}{g'(x)}$ is increasing $($or decreasing$)$ in $(a,b)$, then the functions $\frac{f(x)-f(b)}{g(x)-g(b)}$ and $\frac{f(x)-f(a)}{g(x)-g(a)}$ are also increasing $($or decreasing$)$ in
$(a,b)$.
\end{lem}

The above Lemma~\ref{lem2.5} can be found, for examples, in~\cite[p.~292, Lemma~1]{6}, \cite[p.~57, Lemma~2.3]{jordan-generalized-simp.tex}, \cite[p.~92, Lemma~1]{Gene-Jordan-Inequal.tex}, \cite[p.~161, Lemma~2.3]{jordan-strengthened.tex} and closely related references therein.

\section{Some trigonometric inequalities}

For providing concise proofs of Propositions~\ref{th1.1} to~\ref{th1.4} and finding some new sharp inequalities involving the Seiffert, contra-harmonic, arithmetic, geometric, harmonic, and root-square means of two positive real numbers $a$ and $b$ with $a\ne b$, we need the following monotonicity of some functions involving the sine and cosine functions, which can be proved by making use of Lemmas~\ref{lem2.1} to~\ref{lem2.5}.

\begin{thm}\label{th3.1}
For $x\in(0,\pi)$, the function
\begin{equation}\label{eq3.1-func}
h_1(x)=\frac{\frac{\sin x}{x}-\cos^2x}{\sin^2x}
\end{equation}
is strictly decreasing and has the limits
\begin{equation}\label{eq3.1-func-lim}
\lim_{x \to 0^+}h_1(x)=\frac{5}{6} \quad\text{and}\quad \lim_{x \to\pi^-}h_1(x)=-\infty.
\end{equation}
\end{thm}

\begin{proof}
It is easy to see that
\begin{equation*}
h_1(x)=\frac{1}{x\sin x}-\frac{1}{\sin ^2 x}+1
\end{equation*}
for $x\in (0,\pi)$. By using~\eqref{eq2.1} and~\eqref{eq2.2}, we have
\begin{align*}
h_1(x)&=\frac{1}{x^2}+\sum_{n=1}^{\infty}\frac{2^{2n}-2}{(2n)!}|B_{2n}|x^{2n-2} -\frac{1}{x^2}-\sum_{n=1}^{\infty}\frac{(2n-1)2^{2n}}{(2n)!}|B_{2n}|x^{2n-2}+1\\
    &=\sum_{n=1}^{\infty}\frac{(1-n)2^{2n+1}-2}{(2n)!}|B_{2n}|x^{2n-2}+1.
\end{align*}
So the function $h_1(x)$ is strictly decreasing on $(0,\pi)$.
\par
The two limits in~\eqref{eq3.1-func-lim} come from L'H\^ospital rule and standard argument.
The proof of Theorem~\ref{th3.1} is complete.
\end{proof}

\begin{thm}\label{th3.2}
For $x\in(0,2\pi)$, the function
\begin{equation}\label{eq3.2}
h_2(x)=\frac{\sin x-x\cos x}{x(1-\cos x)}
\end{equation}
is strictly decreasing and has the limits
\begin{equation}\label{lim-thm3.2}
\lim_{x \to 0^+}h_2(x)=\frac{2}{3}\quad\text{and}\quad \lim_{x \to (2\pi)^-}h_2(x)=-\infty.
\end{equation}
\end{thm}

\begin{proof}
Let
\begin{equation*}
f_1(x)=\sin x-x\cos x\quad \text{and}\quad f_2(x)=x(1-\cos x).
\end{equation*}
Then
\begin{equation*}
\frac{f_1'(x)}{f_2'(x)}=\frac{x\sin x}{1-\cos x+x\sin x}
=\biggl(1+\frac{1-\cos x}{x\sin x}\biggr)^{-1}
=\biggl(1+\frac{\tan \frac{x}{2}}{x}\biggr)^{-1}.
\end{equation*}
Since
\begin{equation*}
    \frac{f_2'(x)}{f_1'(x)}=\frac{1-\cos x+x\sin x}{x\sin x}=1+\frac{1-\cos x}{x\sin x}=1+\frac{\tan \frac{x}{2}}{x}
\end{equation*}
is increasing on both $(0,\pi)$ and $(\pi,2\pi)$, the function $\frac{f_1'(x)}{f_2'(x)}$ is decreasing on both $(0,\pi)$ and $(\pi,2\pi)$. Hence, by virtue of Lemma~\ref{lem2.5} and the continuity of $h_2(x)$ at $x=\pi$, it follows that the function $h_2(x)$ is strictly decreasing on $(0,2\pi)$.
\par
Two limits in~\eqref{lim-thm3.2} may be derived from L'H\^ospital rule and standard argument.
The proof of Theorem~\ref{th3.2} is complete.
\end{proof}

\begin{thm}\label{th3.3}
For $x\in(0,\pi)$, the function
\begin{equation}\label{eq3.4}
h_3(x)=\frac{x-\sin x\cos x}{x\sin^2x}
\end{equation}
is strictly increasing and satisfies
\begin{equation}\label{lim-thm3.3}
\lim_{x \to 0^+}h_3(x)=\frac{2}{3}\quad\text{and}\quad \lim_{x \to\pi^-}h_3(x)=\infty.
\end{equation}
\end{thm}

\begin{proof}
The function $h_3(x)$ may be rewritten as
\begin{equation*}
h_3(x)=\frac{1}{\sin ^2 x}-\frac{\cot x}{x}
\end{equation*}
for $x\in (0,\pi)$. By using~\eqref{eq2.2} and~\eqref{eq2.0}, we have
\begin{align*}
h_3(x)&=\frac{1}{x^2}-\sum_{n=1}^{\infty}\frac{2^{2n}(2n-1)}{(2n)!}|B_{2n}|x^{2n-2}-\frac{1}{x^2}-\sum_{n=1}^{\infty}\frac{2^{2n}}{(2n)!}|B_{2n}|x^{2n-2}\\
    &=\sum_{n=1}^{\infty}\frac{n2^{2n+1}}{(2n)!}|B_{2n}|x^{2n-2}.
\end{align*}
So the function $h_3(x)$ is strictly increasing on $(0,\pi)$.
\par
The limits in~\eqref{lim-thm3.3} may be concluded from L'H\^ospital rule and standard argument. The proof of Theorem~\ref{th3.3} is complete.
\end{proof}

\begin{thm}\label{th3.4}
For $x\in(0,\pi)$, the function
\begin{equation}\label{eq3.5}
h_4(x)=\frac{(x-\sin x)\cos x}{x-\sin x\cos x}
\end{equation}
is strictly decreasing, with
\begin{equation}\label{th3.4-lim}
\lim_{x\to0^+}h_4(x)=\frac14\quad\text{and}\quad \lim_{x\to\pi^-}h_4(x)=-1.
\end{equation}
\end{thm}

\begin{proof}
It is obvious that
\begin{equation*}
h_4(x)=1-\frac{f_1(x)}{f_2(x)},
\end{equation*}
where
\begin{equation*}
f_1(x)=x(1-\cos x)\quad \text{and}\quad f_2(x)=x-\sin x\cos x.
\end{equation*}
Easy computations give
\begin{equation*}
\frac{f_1'(x)}{f_2'(x)}=\frac{1-\cos x+x\sin x}{2\sin ^2x}\triangleq\frac{f_3(x)}{f_4(x)}
\end{equation*}
and
\begin{equation*}
\frac{f_3'(x)}{f_4'(x)}=\frac{2\sin x+x\cos x}{4\sin x\cos x}=\frac{1}{2\cos x}+\frac{x}{4\sin x}.
\end{equation*}
Since $\frac{1}{\cos x}$ and $\frac{x}{\sin x}$ are increasing on both $\bigl(0,\frac\pi2\bigr)$ and $\bigl(\frac\pi2,\pi\bigr)$, the function $\frac{f_3'(x)}{f_4'(x)}$ is strictly increasing on both $\bigl(0,\frac\pi2\bigr)$ and $\bigl(\frac\pi2,\pi\bigr)$. Hence, By Lemma~\ref{lem2.5} and the continuity of $h_4(x)$ at $x=\frac\pi2$, we see that $h_4(x)$ is is strictly decreasing on $(0,\pi)$.
\par
The limits in~\eqref{th3.4-lim} can be deduced from L'H\^ospital rule and standard argument.
The proof of Theorem~\ref{th3.4} is complete.
\end{proof}

\section{Concise  proofs of  Propositions~\ref{th1.1} to~\ref{th1.4}}

Now we are in a position to provide concise proofs of Propositions~\ref{th1.1} to~\ref{th1.4}.

\begin{proof}[A concise proof of Proposition~\ref{th1.1}]
The inequality~\eqref{th1.1-ineq} is equivalent to
\begin{equation*}
\alpha<\frac{P(a,b)-H(a,b)}{A(a,b)-H(a,b)}<\beta.
\end{equation*}
Without loss of generality, we assume that $a>b>0$. Let $x=\frac{a}{b}>1$. Then
\begin{align*}
    \frac{P(a,b)-H(a,b)}{A(a,b)-H(a,b)}&=\frac{\frac{x-1}{2\arcsin \frac{x-1}{x+1}}-\frac{2x}{1+x}}{\frac{x+1}{2}-\frac{2x}{1+x}}.
\end{align*}
Let $t=\frac{x-1}{x+1}$. Then $t\in (0,1)$ and
\begin{equation*}
    \frac{P(a,b)-H(a,b)}{A(a,b)-H(a,b)}= \frac{\frac{t}{\arcsin t}-\bigl(1-t^2\bigr)}{t^2}.
\end{equation*}
Let $t=\sin\theta$ for $\theta\in \bigl(0,\frac{\pi}{2}\bigr)$. Then
\begin{equation*}
\frac{P(a,b)-H(a,b)}{A(a,b)-H(a,b)}= \frac{\frac{\sin\theta}{\theta}-\cos^2\theta}{\sin ^2\theta}.
\end{equation*}
By Theorem~\ref{th3.1} and $h_1\bigl(\frac\pi2\bigr)=\frac2\pi$, Proposition~\ref{th1.1} follows.
\end{proof}

\begin{proof}[A concise proof of Proposition~\ref{th1.2}]
The inequality~\eqref{th1.2-ineq} can be rearranged as
\begin{equation*}
\alpha<\frac{P(a,b)-H(a,b)}{C(a,b)-H(a,b)}<\beta.
\end{equation*}
Without loss of generality, we assume that $a>b>0$. Let $x=\frac{a}{b}>1$. Then
\begin{align*}
\frac{P(a,b)-H(a,b)}{C(a,b)-H(a,b)}&=\frac{\frac{x-1}{2\arcsin \frac{x-1}{x+1}}-\frac{2x}{x+1}}{\frac{x^2+1}{x+1}-\frac{2x}{x+1}}.
\end{align*}
Let $t=\frac{x-1}{x+1}$. Then $t\in (0,1)$ and
\begin{equation*}
    \frac{P(a,b)-H(a,b)}{C(a,b)-H(a,b)}= \frac{\frac{t}{\arcsin t}-\bigl(1-t^2\bigr)}{2t^2}.
\end{equation*}
Let $t=\sin\theta$ for $\theta\in \bigl(0,\frac{\pi}{2}\bigr)$. Then
\begin{equation*}
\frac{P(a,b)-H(a,b)}{C(a,b)-H(a,b)}=\frac{\frac{\sin\theta}{\theta}-\cos^2\theta}{2\sin ^2\theta}.
\end{equation*}
By Theorem~\ref{th3.1} and $h_1\bigl(\frac\pi2\bigr)=\frac2\pi$, we obtain Proposition~\ref{th1.2}.
\end{proof}

\begin{proof}[A concise proof of Proposition~\ref{th1.3}]
The inequality~\eqref{th1.3-ineq} may be rewritten as
\begin{equation*}
\alpha<\frac{T(a,b)-A(a,b)}{S(a,b)-A(a,b)}<\beta.
\end{equation*}
Without loss of generality, we assume that $a>b>0$. Let $x=\frac{a}{b}>1$. Then
\begin{equation*}
\frac{T(a,b)-A(a,b)}{S(a,b)-A(a,b)} =\frac{\frac{x-1}{2\arctan\frac{x-1}{x+1}}-\frac{x+1}{2}}{\sqrt{\frac{x^2+1}{2}}-\frac{x+1}{2}}.
\end{equation*}
Let $t=\frac{x-1}{x+1}$. Then $t\in (0,1)$ and
\begin{equation*}
\frac{T(a,b)-A(a,b)}{S(a,b)-A(a,b)}=\frac{\frac{t}{\arctan t}-1}{\sqrt{1+t^2}\,-1}.
\end{equation*}
Let $t=\tan\theta$ for $\theta\in\bigl(0,\frac\pi4\bigr)$. Then
\begin{equation*}
\frac{T(a,b)-A(a,b)}{S(a,b)-A(a,b)}=\frac{\frac{\tan\theta}{\theta}-1}{\frac{1}{\cos\theta}-1}
=\frac{\sin\theta-\theta\cos\theta}{\theta(1-\cos\theta)}.
\end{equation*}
By Theorem~\ref{th3.2} and $h_2\bigl(\frac\pi4\bigr)=\frac{4-\pi}{(\sqrt{2}\,-1)\pi}$, we obtain Proposition~\ref{th1.3}.
\end{proof}

\begin{proof}[A concise proof of Proposition~\ref{th1.4}]
It is clear that the double inequality~\eqref{th1.4-ineq} is equivalent to
\begin{equation*}
\alpha<\frac{P(a,b)-H(a,b)}{\overline{C}(a,b)-H(a,b)}<\beta.
\end{equation*}
Without loss of generality, we assume that $a>b>0$. Let $x=\frac{a}{b}>1$. Then
\begin{equation*}
\frac{P(a,b)-H(a,b)}{\overline{C}(a,b)-H(a,b)} =\frac{\frac{x-1}{2\arcsin\frac{x-1}{x+1}} -\frac{2x}{x+1}}{\frac{2(x^2+x+1)}{3(x+1)}-\frac{2x}{x+1}}.
\end{equation*}
Let $t=\frac{x-1}{x+1}$. Then $t\in (0,1)$ and
\begin{equation*}
\frac{P(a,b)-H(a,b)}{\overline{C}(a,b)-H(a,b)} =\frac{\frac{t}{\arcsin t}-\bigl(1-t^2\bigr)}{\frac{4}{3}t^2}.
\end{equation*}
Let $t=\sin\theta$ for $\theta\in \bigl(0,\frac{\pi}{2}\bigr)$. Then
\begin{equation*}
\frac{P(a,b)-H(a,b)}{\overline{C}(a,b)-H(a,b)}= \frac{\frac{\sin\theta}{\theta}-\cos^2\theta}{\frac{4}{3}\sin ^2\theta}.
\end{equation*}
By Theorem~\ref{th3.1} and $h_1\bigl(\frac\pi2\bigr)=\frac2\pi$, we obtain Proposition~\ref{th1.4}.
\end{proof}

\begin{rem}
From the above concise proofs, we can conclude that Propositions~\ref{th1.1}, \ref{th1.2}, and~\ref{th1.4} are equivalent to each other.
\end{rem}

\section{New inequalities involving Seiffert and other means}
In this section we will find some new sharp inequalities involving the Seiffert, contra-harmonic, arithmetic, geometric, harmonic, and root-square means of two positive real numbers $a$ and $b$ with $a\ne b$.

\begin{thm}\label{th5.1}
The double inequality
\begin{equation}\label{th5.1-ineq}
\alpha C(a,b)+(1-\alpha)H(a,b)<T(a,b)<\beta C(a,b)+(1-\beta)H(a,b)
\end{equation}
holds for all $a,b>0$ with $a\ne b$ if and only if $\alpha\le \frac{2}{\pi}$ and $\beta\ge \frac{2}{3}$.
\end{thm}

\begin{proof}
The double inequality~\eqref{th5.1-ineq} is the same as
\begin{equation*}
\alpha<\frac{T(a,b)-H(a,b)}{C(a,b)-H(a,b)}<\beta.
\end{equation*}
Without loss of generality, we assume that $a>b>0$. Let $x=\frac{a}{b}>1$. Then
\begin{equation*}
\frac{T(a,b)-H(a,b)}{C(a,b)-H(a,b)} =\frac{\frac{x-1}{2\arctan\frac{x-1}{x+1}} -\frac{2x}{x+1}}{\frac{x^2+1}{x+1}-\frac{2x}{x+1}}.
\end{equation*}
Let $t=\frac{x-1}{x+1}$. Then $t\in (0,1)$ and
\begin{equation*}
    \frac{T(a,b)-H(a,b)}{C(a,b)-H(a,b)}=\frac{\frac{t}{\arctan t}-1+t^2}{2t^2}.
\end{equation*}
Let $t=\tan\theta$ for $\theta\in \bigl(0,\frac{\pi}{4}\bigr)$. Then
\begin{equation*}
\frac{T(a,b)-H(a,b)}{C(a,b)-H(a,b)}= \frac{\frac{\tan\theta}{\theta}-1+\tan^2\theta}{2\tan^2\theta}=1- \frac{\theta-\sin\theta\cos\theta }{2\theta\sin^2\theta}.
\end{equation*}
By Theorem~\ref{th3.3} and $h_3\bigl(\frac\pi4\bigr)=2-\frac4\pi$, we obtain Theorem~\ref{th5.1}.
\end{proof}

\begin{thm}\label{th5.2}
The double inequality
\begin{equation}
\alpha C(a,b)+(1-\alpha)T(a,b)<S(a,b)<\beta C(a,b)+(1-\beta)T(a,b)
\end{equation}
holds for all $a,b>0$ with $a\ne b$ if and only if $\alpha\le \frac{\pi-2\sqrt{2}\,}{\sqrt{2}\,\pi-2\sqrt{2}\,}$ and $\beta\ge \frac{1}{4}$.
\end{thm}

\begin{proof}
It is sufficient to show
\begin{equation*}
\alpha<\frac{S(a,b)-T(a,b)}{C(a,b)-T(a,b)}<\beta.
\end{equation*}
Without loss of generality, we assume $a>b>0$. Let $x=\frac{a}{b}>1$. Then
\begin{equation*}
\frac{S(a,b)-T(a,b)}{C(a,b)-T(a,b)}=\frac{\sqrt{\frac{x^2+1}{2}}\, -\frac{x-1}{2\arctan\frac{x-1}{x+1}}}{\frac{x^2+1}{x+1}-\frac{x-1}{2\arctan\frac{x-1}{x+1}}}.
\end{equation*}
Let $t=\frac{x-1}{x+1}$. Then $t\in (0,1)$ and
\begin{equation*}
\frac{S(a,b)-T(a,b)}{C(a,b)-T(a,b)}=\frac{\sqrt{1+t^2}\,-\frac{t}{\arctan t}}{1+t^2-\frac{t}{\arctan t}}.
\end{equation*}
Let $t=\tan\theta$ for $\theta\in\bigl(0,\frac\pi4\bigr)$. Then
\begin{equation*}
\frac{\frac{1}{\cos\theta}-\frac{\tan\theta}{\theta}}{\frac{1}{\cos ^2\theta}-\frac{\tan\theta}{\theta}}=\frac{\cos\theta(\theta-\sin\theta)}{\theta-\sin\theta\cos\theta}.
\end{equation*}
By Theorem~\ref{th3.4} and $h_4\bigl(\frac\pi4\bigr)=\frac{\pi-2\sqrt{2}\,}{\sqrt{2}\,\pi-2\sqrt{2}\,}$, we obtain Theorem~\ref{th5.2}.
\end{proof}

\begin{thm}\label{th5.3}
The double inequality
\begin{equation}\label{th5.3-ineq}
\alpha A(a,b)+(1-\alpha)G(a,b)<P(a,b)<\beta A(a,b)+(1-\beta)G(a,b)
\end{equation}
holds for all $a,b>0$ with $a\ne b$ if and only if $\alpha\le \frac{2}{\pi}$ and $\beta\ge \frac{2}{3}$.
\end{thm}

\begin{proof}
The inequality~\eqref{th5.3-ineq} is equivalent to
\begin{equation*}
\alpha<\frac{P(a,b)-G(a,b)}{A(a,b)-G(a,b)}<\beta.
\end{equation*}
Without loss of generality, assume $a>b>0$. Let $x=\frac{a}{b}$. Then $x>1$ and
\begin{equation*}
   \frac{P(a,b)-G(a,b)}{A(a,b)-G(a,b)}=\frac{\frac{x-1}{2\arcsin \frac{x-1}{x+1}}-\sqrt{x}\,}{\frac{x+1}{2}-\sqrt{x}\,}.
\end{equation*}
Let $t=\frac{x-1}{x+1}$. Then $t\in (0,1)$ and
\begin{equation*}
   \frac{P(a,b)-G(a,b)}{A(a,b)-G(a,b)}= \frac{\frac{t}{\arcsin t}-\sqrt{1-t^2}\,}{1-\sqrt{1-t^2}\,}.
\end{equation*}
Let $t=\sin\theta$ for $\theta\in \bigl(0,\frac{\pi}{2}\bigr)$. Then
\begin{equation}\label{5.6}
  \frac{P(a,b)-G(a,b)}{A(a,b)-G(a,b)}=  \frac{\frac{\sin\theta}{\theta}-\cos\theta}{1-\cos\theta}=  \frac{\sin\theta-\theta\cos\theta}{\theta(1-\cos\theta)}.
\end{equation}
By Theorem~\ref{th3.2} and $h_2\bigl(\frac\pi2\bigr)=\frac2\pi$, we obtain Theorem~\ref{th5.3}.
\end{proof}

\begin{rem}
In \cite{back23}, the double inequality
\begin{equation*}
\frac{1}{2}[A(a,b)+G(a,b)]<P(a,b)<\frac{2}{3}A(a,b)+\frac{1}{3}G(a,b)
\end{equation*}
for all $a,b>0$ with $a\ne b$, a special case of Theorem~\ref{th5.3} for $\alpha=\frac12$ and $\beta=\frac23$,  was given.
\end{rem}

\end{document}